\documentclass[11pt,reqno]{amsart}
\usepackage{amssymb,latexsym}
\usepackage{graphicx}
\usepackage{url}		
\usepackage{float}

\calclayout

\makeatletter
\newcommand{\monthyear}[1]{%
  \def\@monthyear{\uppercase{#1}}}
\newcommand{\volnumber}[1]{%
  \def\@volnumber{\uppercase{#1}}}
\AtBeginDocument{%
\def\ps@plain{\ps@empty
  \def\@oddfoot{\@monthyear \hfil \thepage}%
  \def\@evenfoot{\thepage \hfil \@volnumber}}
\def\ps@firstpage{\ps@plain}
\def\ps@headings{\ps@empty
  \def\@evenhead{%
    \setTrue{runhead}%
    \def\thanks{\protect\thanks@warning}%
    \uppercase{The Fibonacci Quarterly}\hfil}%
  \def\@oddhead{%
    \setTrue{runhead}%
    \def\thanks{\protect\thanks@warning}%
    \hfill\uppercase{Hypergeometric Template}}%
  \let\@mkboth\markboth
  \def\@evenfoot{%
    \thepage \hfil \@volnumber}%
  \def\@oddfoot{%
    \@monthyear \hfil \thepage}%
  }%
\footskip=25pt
\pagestyle{headings}%
}
\makeatother

\newcommand{\Rh}{{\mathcal R}}

\theoremstyle{plain}
\numberwithin{equation}{section}
\newtheorem{thm}{Theorem}[section]
\newtheorem{corollary}[thm]{Corollary}
\newtheorem{theorem}[thm]{Theorem}
\newtheorem{lemma}[thm]{Lemma}
\newtheorem{example}[thm]{Example}

\begin{document}
\monthyear{Month Year}
\volnumber{Volume, Number}
\setcounter{page}{1}

\title{The Pascal Rhombus and the Generalized Grand Motzkin Paths}

\author{Jos\'e L. Ram\'{\i}rez }
\address{\noindent Departamento  de Matem\'aticas, Universidad Sergio Arboleda, Bogot\'a,  Colombia}
\email{josel.ramirez@ima.usergioarboleda.edu.co}

\begin{abstract}
In the present article, we find a closed expression  for the entries of the Pascal rhombus. Moreover, we show a relation between the entries of the Pascal rhombus and a family of generalized grand Motzkin paths.
\end{abstract}

\maketitle

\section{Introduction}

The Pascal rhombus was introduced by Klostermeyer et al. \cite{Klos} as a variation of the well-known Pascal triangle. It is an infinite array $\Rh=[r_{i,j}]_{i=0, j=-\infty}^{\infty, \infty}$ defined by
\begin{align}\label{recrho}
r_{i,j}=r_{i-1,j-1}+r_{i-1,j}+r_{i-1,j+1}+r_{i-2,j}, \quad i\geq 2, \ \  j\in\mathbb{Z},
\end{align}
with the initial conditions
$$r_{0,0}=r_{1,-1}=r_{1,0}=r_{1,1}=1, \quad r_{0,j}=0 \ (j\neq 0), \quad r_{1,j}=0, \ (j\neq -1, 0, 1). $$

The first few rows of $\Rh$ are

 \begin{table}[h]
  \centering
  \begin{tabular}{ccccccccccccccc}
 & & &  &  & 1& &  &  & &\\
 & & &  & 1 & 1 & 1 & & & &\\
 & & & 1 & 2 & 4 & 2 & 1 & & &\\
& &1 & 3 &  8 &  9&  8 &  3 & 1 & &\\
& 1& 4& 13 & 22 & 29 & 22& 13 & 4& 1 &\\
1& 5 &19 &  42 & 72 & 82  & 72 &  42 & 19 & 5 &1\\
$\vdots$ & &$\vdots$ & &  $\vdots$ & &$\vdots$ & &$\vdots$ & & $\vdots$
   \end{tabular}
  \caption{Pascal Rhombus.}
  \label{rombo2}
\end{table}
Klostermeyer et al. \cite{Klos} studied several identities of the Pascal rhombus. Goldwasser et al. \cite{Gold} proved that the limiting ratio of the number of ones to the number of zeros in $\Rh$, taken modulo 2, approaches zero. This result was generalized by Mosche \cite{Moshe}. Recently,  Stockmeyer \cite{Sto} proved
 four conjectures about the Pascal rhombus modulo 2  given in \cite{Klos}.\\

The Pascal rhombus corresponds with the entry A059317 in the On-Line Encyclopedia of Integer Sequences (OEIS) \cite{OEIS}, where it is possible to read: \emph{There does not seem to be a simple expression for $r_{i,j}$}. \\

In the present article, we find an explicit expression for $r_{i,j}$. In particular, we prove that
$$r_{i,j}=\sum_{m=0}^{i}\sum_{l=0}^{i-j-2m}\binom{2m+j}{m}\binom{l+j+2m}{l}\binom{l}{i-j-2m-l}.$$
For this we show that $r_{i,j}$ is equal to the number of 2-generalized grand Motzkin paths.

\section{The Main Result}

A \emph{Motzkin path} of length $n$ is a lattice path of  $\mathbb{Z \times Z}$ running from $(0, 0)$ to $(n, 0)$ that never passes below the $x$-axis and whose permitted steps are the up diagonal step $U=(1, 1)$, the down diagonal step $D=(1,-1)$ and the horizontal step $H=(1, 0)$, called rise, fall and level step, respectively.  The number of Motzkin paths of length $n$ is the $n$-th \emph{Motzkin number} $m_{n}$, (sequence  A001006).  Many other examples of  bijections between Motzkin numbers and others combinatorial objects  can be found in \cite{FRA}. A  \emph{grand Motzkin path} of length  $n$ is a Motzkin path without the condition that never passes below the $x$-axis.  The number of grand Motzkin paths of length $n$ is the $n$-th \emph{grand Motzkin number} $g_{n}$, sequence A002426.  A $2$-\emph{generalized Motzkin path} is a Motzkin path with an additional step $H_2=(2,0)$. The number of $2$-generalized Motzkin paths of length  $n$ is denoted by $m^{(2)}_{n}$. Analogously, we have $2$-\emph{grand generalized Motzkin paths}, and the number of these paths of length $n$ is denoted by $g^{(2)}_{n}$.

\begin{lemma}\label{lem2}
The generating function of the 2-generalized Motzkin numbers is given by
\begin{align}
B(x):=\sum_{i=0}^\infty m_i^{(2)} x^i&=\frac{1-x-x^2-\sqrt{1 - 2 x - 5x^2 + 2x^3 + x^4}}{2x^2}\label{ida}\\
&=\frac{F(x)}{x}C(F(x)^2)\label{idb},
\end{align}
where $F(x)$ and $C(x)$ are the generating functions of the Fibonacci numbers and Catalan numbers, i.e.,
$$F(x)=\frac{x}{1-x-x^2}, \quad C(x)=\frac{1-\sqrt{1-4x}}{2x}.$$
\end{lemma}
\begin{proof}
From the first return decomposition any nonempty  2-generalized Motzkin path $T$  may be decomposed as either $UT'DT'', \  HT',$  or $H_2T'$, where $T', T''$ are 2-generalized Motzkin paths (possible empty).
Making use of the Flajolet's symbolic method (cf. \cite{Flaj}) we obtain
\begin{align*}
B(x)=1+(x+x^2)B(x)+x^2B(x)^2.
\end{align*}
Therefore Equation (\ref{ida}) follows.  Moreover,
\begin{align*}
B(x)&=\frac{1-x-x^2-\sqrt{(1-x-x^2)^2-4x^2}}{2x^2}=\frac{1-\sqrt{1-4\left(\frac{x}{1-x-x^2}\right)^2}}{\frac{2x^2}{1-x-x^2}}\\
&=\frac{1}{1-x-x^2} \frac{1-\sqrt{1-4F(x)^2}}{2F(x)^2}=\frac{F(x)}{x}C(F(x)^2).
\end{align*}
\end{proof}

The height of a 2-generalized grand Motzkin path is defined as the final height of the path, i.e., the stopping $y$-coordinate. The number of 2-generalized grand Motzkin paths of length $n$ and height $j$ is denoted by $g_{n,j}^{(2)}$.

\begin{theorem}
The generating function of the  2-generalized grand Motzkin paths of  height $j$ is
\begin{align*}
M^{(j)}(x):=\sum_{i=0}^{\infty}g_{i,j}^{(2)}x^i=\frac{F(x)^{j+1}C(F(x)^2)^j}{x(1-2F(x)^2C(F(x)^2))},
\end{align*}
where $F(x)$ and $C(x)$ are the generating function of the Fibonacci numbers and Catalan numbers. Moreover,
\begin{align*}
g_{i,j}^{(2)}=\sum_{m=0}^{i}\sum_{l=0}^{i-j-2m}\binom{2m+j}{m}\binom{l+j+2m}{l}\binom{l}{i-j-2m-l},\quad (0\leq j\leq i).
\end{align*}
\end{theorem}
\begin{proof}
Consider any 2-generalized grand Motzkin path $P$. Then any nonempty path $P$ may be decomposed as either
$$UMDP', \quad DMUP', \quad HP', \quad H_2P', \quad \text{or} \quad UM_1UM_2\cdots UM_j,$$
where $M, M_1, \dots, M_j$ are 2-generalized Motzkin paths (possible empty), $P'$ is a 2-generalized grand Motzkin path (possible empty).

Schematically,
\begin{figure}[H]
\centering
\includegraphics[scale=0.7]{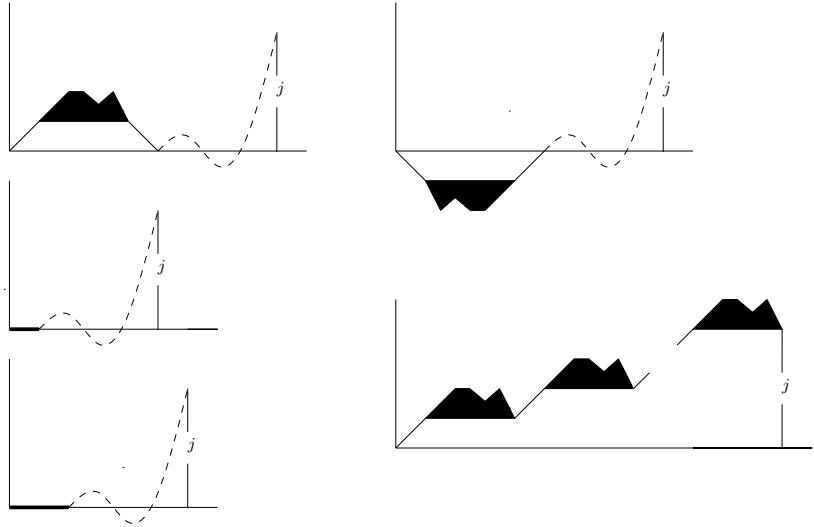}
\caption{Factorizations of any 2-generalized grand Motzkin path.} \label{gmot}
\end{figure}

From the  Flajolet's symbolic method  we obtain

$$M^{(j)}(x)=2x^2B(x)M^{(j)}(x) + (x+x^2)M^{(j)}(x)+x^j(B(x))^j, \quad j\geq 0.$$
Therefore
$$M^{(j)}(x)=\frac{x^jB(x)^j}{1-x-x^2-2x^2B(x)}.$$
 From Lemma \ref{lem2} we get
\begin{align*}
M^{(j)}(x)=\frac{x^j \left(\frac{F(x)}{x}C(F(x)^2)\right)^j}{1-x-x^2-2x^2\frac{F(x)}{x}C(F(x)^2)}=\frac{F(x)^{j+1}C(F(x)^2)^j}{x(1-2F(x)^2C(F(x)^2))}.
\end{align*}
On the other hand,  from the following identity (see Ec. 2.5.15 of \cite{Wilf})
$$\frac{1}{\sqrt{1-4x}}\left(\frac{1-\sqrt{1-4x}}{x}\right)^k=\sum_{m=0}^\infty\binom{2m+k}{m}x^m$$
we obtain
$$\frac{C(x^2)^j}{1-2x^2C(x^2)}=\sum_{m=0}^\infty\binom{2m+j}{m}x^{2m}$$
Therefore
\begin{align*}
M^{(j)}(x)&=\frac{F(x)^{j+1}(x)}{x}\sum_{m=0}^{\infty}\binom{2m+j}{m}F(x)^{2m}=\frac{1}{1-x-x^2}\sum_{m=0}^{\infty}\binom{2m+j}{m}F(x)^{2m+j}\\
&=\sum_{m=0}^{\infty}\binom{2m+j}{m}\frac{x^{2m+j}}{(1-x-x^2)^{2m+j+1}}=\sum_{m=0}^{\infty}\sum_{l=0}^{\infty}\binom{2m+j}{m}\binom{l+j+2m}{l}(1+x)^lx^{2m+j+l}\\
&=\sum_{m=0}^{\infty}\sum_{l=0}^{\infty}\sum_{s=0}^{l}\binom{2m+j}{m}\binom{l+j+2m}{l}\binom{l}{s}x^{2m+j+l+s},
\end{align*}
Put $t=2m+j+l+s$
\begin{align*}
M^{(j)}(x)&=\sum_{m=0}^{\infty}\sum_{l=0}^{\infty}\sum_{t=2m+j+l}^{2m+j+2l}\binom{2m+j}{m}\binom{l+j+2m}{l}\binom{l}{t-2m-j-l}x^t.
\end{align*}
The result follows by comparing the coefficients.
\end{proof}

\begin{theorem}\label{rel}
The number of 2-generalized grand Motzkin paths of length $n$ and height $j$ is equal to the entry $(n,j)$ in the Pascal rhombus, i.e.,
$$r_{n,j}=g_{n,j}^{(2)}.$$
\end{theorem}
\begin{proof}
The  sequence $g_{n,j}^{(2)}$ satisfies the recurrence (\ref{recrho}) and the same initial values. It is clear,
by considering the positions preceding to the last step of any  2-generalized grand Motzkin path.
\end{proof}

\begin{corollary}
The generating function of the $j$th column of the Pascal rhombus is
\begin{align*}
L_j(x)=\frac{F(x)^{j+1}C(F(x)^2)^j}{x(1-2F(x)^2C(F(x)^2))},
\end{align*}
where $F(x)$ and $C(x)$ are the generating function of the Fibonacci numbers and Catalan numbers. Moreover,
\begin{align*}
r_{i,j}=\sum_{m=0}^{i}\sum_{l=0}^{i-j-2m}\binom{2m+j}{m}\binom{l+j+2m}{l}\binom{l}{i-j-2m-l} \quad (0\leq j\leq i).
\end{align*}
\end{corollary}

The convolved Fibonacci numbers $F_j^{(r)}$ are defined by $$(1-x-x^2)^{-r}=\sum_{j=0}^{\infty}F_{j+1}^{(r)}x^j,  \ \ \ r\in \mathbb{Z}^+.$$  If $r=1$ we have the classical Fibonacci sequence.

Note that
\begin{align*}
F_{m+1}^{(r)}=\sum_{j_1+j_2+\cdots +j_r=m}F_{j_1+1}F_{j_2+1} \cdots F_{j_r+1}.
\end{align*}
Moreover, using a result of Gould \cite[p. 699]{GOU} on Humbert polynomials (with $n = j, m = 2,
x = 1/2, y = -1, p = -r$ and $C = 1$), we have
\begin{align*}
F_{j+1}^{(r)}=\sum_{l=0}^{\lfloor j/2 \rfloor}\binom{j+r-l-1}{j-l}\binom{j-l}{l}.
\end{align*}

\begin{corollary}The following equality holds
\begin{align}
r_{i,j}=\sum_{m=0}^{\lfloor\frac{i-j}{2}\rfloor}\binom{2m+j}{m}F_{i-j-2m+1}^{(j+2m+1)},\label{rombo3}
\end{align}
where $F_{l}^{(r)}$ are the convolved Fibonacci numbers.
\end{corollary}
\begin{proof}
\begin{align*}
L_n(x)&=\sum_{m=0}^{\infty}\binom{2m+n}{m}\frac{x^{2m+n}}{(1-x-x^2)^{n+2m+1}}=\sum_{m=0}^{\infty}\sum_{j=0}^{\infty} \binom{2m+n}{m}F_{j+1}^{(n+2m+1)}x^{2m+n+j},
\end{align*}
Put $t=2m+n+j$
\begin{align*}
L_n(x)&=\sum_{m=0}^{\infty}\sum_{t=2m+n}^{\infty} \binom{2m+n}{m}F_{t-2m-n+1}^{(n+2m+1)}x^{t}.
\end{align*}
The result follows by comparing the coefficients.
\end{proof}

\begin{example}
The generating function of the central column of the Pascal rhombus   (sequence A059345) is
$$L_0(x)=\frac{1}{\sqrt{1 - 2 x - 5 x^2 + 2 x^3 + x^4}}=1 + x + 4 x^2 + 9 x^3 + 29 x^4 + 82 x^5 + 255 x^6  +\cdots.$$

The generating function of the first few columns ($j=1, 2, 3$) of the Pascal rhombus  are:
\begin{align*}
L_1(x)&=x + 2 x^2 + 8 x^3 + 22 x^4 + 72 x^5 + 218 x^6 + 691 x^7 + 2158 x^8  +\cdots, \quad \text{(A106053)}\\
L_2(x)&=x^2 + 3 x^3 + 13 x^4 + 42 x^5 + 146 x^6 + 476 x^7 + 1574 x^8 + \cdots, \quad \text{(A106050)}\\
 L_3(x)&=x^3 + 4 x^4 + 19 x^5 + 70 x^6 + 261 x^7 + 914 x^8 + 3177 x^9  + \cdots
 \end{align*}
\end{example}

\textbf{Remark:} The results of this article were discovered by using the Counting Automata Methodology, \cite{Ram}.

\section{Acknowledgements}

The author thanks the anonymous referee for his/her comments and remarks, which helped  to improve the article.


\medskip

\noindent MSC2010: 05A19, 11B39, 11B37.


\begin{thebibliography}{99}


\bibitem{FRA} F. Bernhart, \emph{Catalan, Motzkin, and Riordan numbers}, Discrete Math. \textbf{204}(1999), 73--112.

\bibitem{Ram} R. De Castro, A. Ramírez and J. Ramírez, \emph{Applications in enumerative combinatorics of infinite weighted automata and graphs}, Sci. Ann. Comput. Sci., \textbf{24.1} (2014), 137--171.

\bibitem{Flaj} P. Flajolet and R. Sedgewick, \emph{Analytic  Combinatorics}, Camdridge, 2009.

\bibitem{Gold} J. Goldwasser, W. F. Klostermeyer, M. E. Mays and  G. Trapp, \emph{The
density of ones in Pascal's rhombus}, Discrete Math. \textbf{204} (1999), 231--236.

\bibitem{GOU} H. W. Gould, \emph{Inverse series relations and other expansions involving Humbert polynomials}, Duke Math. J.  \textbf{32.4} (1965), 697-711.

\bibitem{Klos} W. F. Klostermeyer, M. E. Mays, L. Soltes and  G. Trapp, \emph{A Pascal rhombus}. The  Fibonacci Quarterly, \textbf{35} (1997), 318--328.

\bibitem{Moshe} Y. Moshe. \emph{The density of $0$'s in recurrence double sequences}. J. Number Theory, \textbf{103} (2003), 109--121.


\bibitem{OEIS}
OEIS Foundation Inc. (2011), The On-Line Encyclopedia of Integer Sequences, \url{http://oeis.org}.

\bibitem{Sto} P. K. Stockmeyer, \emph{The Pascal rhombus and the stealth configuration}, \url{http://arxiv.org/abs/1504.04404}, (2015).




\bibitem{Wilf} H. S. Wilf, \emph{generatingfunctionology}, Academic Press, Second Edition,  1994.




\end{thebibliography}
\end{document}